\documentclass[12pt]{article}
\usepackage{graphicx}
\usepackage{amsmath,amsthm,amssymb,enumerate}%, esint}
\usepackage{euscript,mathrsfs}
\usepackage[left=2cm,right=2cm,top=3.5cm,bottom=3.5cm]{geometry}
\usepackage{color}
\catcode`\@=11 \@addtoreset{equation}{section}

\catcode`\@=12

\allowdisplaybreaks

\newtheorem{Theorem}{Theorem}[section]
\newtheorem{Proposition}[Theorem]{Proposition}
\newtheorem{Lemma}[Theorem]{Lemma}
\newtheorem{Corollary}[Theorem]{Corollary}

\theoremstyle{definition}
\newtheorem{Definition}[Theorem]{Definition}

\newtheorem{Remark}[Theorem]{Remark}

\newcommand{\bTheorem}[1]{
\begin{Theorem} \label{T#1} }
\newcommand{\eT}{\end{Theorem}}

\newcommand{\bProposition}[1]{
\begin{Proposition} \label{P#1}}
\newcommand{\eP}{\end{Proposition}}

\newcommand{\bLemma}[1]{
\begin{Lemma} \label{L#1} }
\newcommand{\eL}{\end{Lemma}}

\newcommand{\bCorollary}[1]{
\begin{Corollary} \label{C#1} }
\newcommand{\eC}{\end{Corollary}}

\newcommand{\bRemark}[1]{
\begin{Remark} \label{R#1} }
\newcommand{\eR}{\end{Remark}}

\newcommand{\bDefinition}[1]{
\begin{Definition} \label{D#1} }
\newcommand{\eD}{\end{Definition}}
\newcommand{\vve}{\vc{v}_\ep}
\newcommand{\Nu}{\mathcal{V}_{t,x}}
\newcommand{\Du}{\mathbb{D}_u}
\newcommand{\DT}{\vc{D}_\vt}

\newcommand{\SGrad}{\mathbb{D}}

\newcommand{\bfphi}{\boldsymbol{\varphi}}

\newcommand{\bFormula}[1]{
\begin{equation} \label{#1}}
\newcommand{\eF}{\end{equation}}

\newcommand{\Ov}[1]{\overline{#1}}

\newcommand{\DC}{C^\infty_c}

\newcommand{\aleq}{\stackrel{<}{\sim}}
\newcommand{\ageq}{\stackrel{>}{\sim}}

\newcommand{\vr}{\varrho}
\newcommand{\vre}{\vr_\ep}

\newcommand{\vte}{\vt_\ep}
\newcommand{\vue}{\vu_\ep}
\newcommand{\tvr}{\tilde \vr}
\newcommand{\tvu}{{\tilde \vu}}
\newcommand{\tvt}{\tilde \vt}

\newcommand{\vt}{\vartheta}
\newcommand{\vu}{\vc{u}}

\newcommand{\vc}[1]{{\bf #1}}

\newcommand{\Div}{{\rm div}_x}
\newcommand{\Grad}{\nabla_x}

\newcommand{\dx}{\,{\rm d} {x}}

\newcommand{\dt}{\,{\rm d} t }

\newcommand{\intO}[1]{\int_{\Omega} #1 \ \dx}

\newcommand{\vv}{\vc{v}}

\newcommand{\ep}{\varepsilon}

\definecolor{Cgrey}{rgb}{0.85,0.85,0.85}
\definecolor{Cblue}{rgb}{0.50,0.85,0.85}
\definecolor{Cred}{rgb}{1,0,0}
\definecolor{fancy}{rgb}{0.10,0.85,0.10}

\newcommand\Cbox[2]{%
    \newbox\contentbox%
    \newbox\bkgdbox%
    \setbox\contentbox\hbox to \hsize{%
        \vtop{
            \kern\columnsep
            \hbox to \hsize{%
                \kern\columnsep%
                \advance\hsize by -2\columnsep%
                \setlength{\textwidth}{\hsize}%
                \vbox{
                    \parskip=\baselineskip
                    \parindent=0bp
                    #2
                }%
                \kern\columnsep%
            }%
            \kern\columnsep%
        }%
    }%
    \setbox\bkgdbox\vbox{
        \color{#1}
        \hrule width  \wd\contentbox %
               height \ht\contentbox %
               depth  \dp\contentbox
        \color{black}
    }%
    \wd\bkgdbox=0bp%
    \vbox{\hbox to \hsize{\box\bkgdbox\box\contentbox}}%
    \vskip\baselineskip%
}

\date{}

\begin{document}

%%%%%%%%%%%%%%%%%%%%%%%%%%%%%%%%

\title{Stability of strong solutions to the Navier--Stokes--Fourier system}

\author{Jan B\v rezina \and Eduard Feireisl
\thanks{The research of E.F.~leading to these results has received funding from the
Czech Sciences Foundation (GA\v CR), Grant Agreement
18--05974S. The Institute of Mathematics of the Academy of Sciences of
the Czech Republic is supported by RVO:67985840.} \and Anton\' \i n Novotn\' y
}

\date{\today}

\maketitle

\bigskip

\centerline{Tokyo Institute of Technology}

\centerline{ 2-12-1 Ookayama, Meguro-ku, Tokyo, 152-8550, Japan}

\bigskip

\centerline{Institute of Mathematics of the Academy of Sciences of the Czech Republic}

\centerline{\v Zitn\' a 25, CZ-115 67 Praha 1, Czech Republic}

\bigskip

\centerline{Universit{\' e} de Toulon, IMATH, EA 2134}

\centerline{BP 20132, 83957 La Garde, France}

\bigskip

\begin{abstract}

We identify a large class of objects - \emph{dissipative measure--valued (DMV) solutions} to the Navier--Stokes--Fourier system - in which the strong solutions are stable. More precisely, a DMV solution coincides with the strong solution emanating from the same initial data as long as the latter exists. The DMV solutions are represented by parameterized families of measures satisfying certain compatibility conditions. They can be seen as an analogue to the dissipative measure--valued solutions introduced earlier in the context of the (inviscid) Euler system.

\end{abstract}


%\tableofcontents

\section{Navier--Stokes--Fourier system}
\label{N}

Motivated by our earlier work on the dissipative measure--valued solutions to the Euler system
(see \cite{BreFei17B}, \cite{BreFei17}, \cite{BreFei17A}), we consider the Navier--Stokes--Fourier system describing the time evolution of a general compressible, viscous, and heat conducting fluid. Our goal is to identify the largest class possible of objects that can be understood as
``solutions'' to this problem. These objects represent a generalization of classical solutions and will coincide with a classical solution emanating from the same initial data as long as the classical solution exists. They will be described by means of a parameterized family of probability measures $\{ \Nu \}_{(t,x) \in Q_T}$, where the vectors $(t,x)$ range over the physical space - the space--time cylinder
\[
Q_T = (0,T) \times \Omega, \ \Omega \subset R^N \ \mbox{a domain},\ N=1,2,3.
\]
We call these objects \emph{dissipative measure--valued (DMV) solutions} although $\{ \Nu \}_{(t,x) \in Q_T}$ may not, in general, coincide with some
Young measure generated by a weakly convergent family of state variables.

The concept of measure--valued solutions in the context of \emph{inviscid} fluids has been used by several authors starting with the pioneering papers by DiPerna \cite{Dip2}
and DiPerna and Majda \cite{DipMaj}, see also related results by Kroener and Zajaczkowski \cite{KrZa} , Matu{\v s}{\accent23u}-Ne{\v c}asov{\' a} and Novotn{\' y}
\cite{MNNO}, Neustupa \cite{Neustup}. A nice summary of the ``fluid mechanics'' applications of this method is in the monograph by M\' alek et al. \cite{MNRR}.
Recently, Tzavaras and coauthors adapted these ideas to problems in elastodynamics, thermoelasticity, and other related problems \cite{CrGaTz}, \cite{DeStTz}, \cite{LatTza}.
Another application to the {C}ucker--{S}male equation is due to Mucha and Peszek
\cite{MuPe}.
Interesting results in numerical analysis have been obtained by
Fjordholm et al. \cite{FjKaMiTa}, \cite{FjMiTa2}, \cite{FjMiTa1}.
Here we apply the same technique to a model of \emph{viscous} fluids. Accordingly, our main ambition is not to capture possible
oscillations in the solution set which might be a highly unlikely phenomenon when viscosity is present, but rather to identify the largest class of objects in which the strong solutions represent a stable set. In particular, the dissipative measure--valued solutions introduced below will satisfy the weak--strong uniqueness principle. We expect these results will find applications whenever approximate solutions are constructed, in particular by means of a numerical scheme. It might be relatively easy to show that families of numerical approximations
generate a measure--valued solution, the latter being quite general object. Such an observation together with the weak--strong uniqueness principle will finally provide a rigorous
convergence proof as long as the strong solution is available. This philosophy has been successfully applied in \cite{FeiLM} in the case of a simpler barotropic fluid flow. We expect that
the results obtained in the present paper will give rise to similar applications in the context of physically relevant models of compressible, viscous, and heat conducting fluids.

\subsection{Phase variables}

The state of the fluid is determined by the \emph{standard phase variables}: The mass density $\vr$, the (absolute) temperature $\vt$, and the (bulk) velocity $\vu \in R^N$. In addition, in the context of viscous and heat conducting fluid, it is convenient to include $\Grad \vu$ and $\Grad \vt$ in the family of state variables.

\subsection{Thermodynamic functions}

The constitutive properties of a given fluid will be expressed through \emph{thermodynamic functions}: The pressure $p = p(\vr, \vt)$, the internal energy $e = e(\vr, \vt)$, and the entropy $s = s(\vr, \vt)$, interrelated through \emph{Gibbs' equation}:
\begin{equation} \label{N1}
\vt D s = D e + p D \left( \frac{1}{\vr} \right).
\end{equation}

\subsection{Physical space}

We use the \emph{Eulerian coordinate system}, where the independent variables are the time $t \in [0,T]$, and the spatial position $x \in \Omega$, $\Omega \subset R^N$ a domain.

\subsection{Field equations}

The time evolution of $\vr = \vr(t,x)$, $\vt = \vt(t,x)$, and $\vu = \vu(t,x)$ is determined by a system of partial differential equations expressing the basic physical principles:

\begin{itemize}
\item {\bf Mass conservation or equation of continuity}
\begin{equation} \label{N2}
\partial_t \vr + \Div (\vr \vu) = 0;
\end{equation}

\item {\bf Conservation of linear momentum or Newton's second law}
\begin{equation} \label{N3}
\partial_t (\vr \vu) + \Div (\vr \vu \otimes \vu) + \Grad p(\vr, \vt) = \Div \mathbb{S},
\end{equation}
where $\mathbb{S}$ is the viscous stress tensor;

\item {\bf Conservation of internal energy}

\begin{equation} \label{N4}
\partial_t (\vr e(\vr, \vt)) + \Div (\vr e(\vr, \vt) \vu ) + \Div \vc{q} = \mathbb{S} : \Grad \vu - p(\vr, \vt) \Div \vu,
\end{equation}
where $\vc{q}$ is the internal energy flux.

\end{itemize}

\subsection{Diffusion terms}

To close the system (\ref{N2}--\ref{N4}), we suppose that $\mathbb{S} = \mathbb{S}(\vr, \vt, \Grad \vu)$ is given by \emph{Newton's rheological law}
\begin{equation} \label{N5}
\mathbb{S} = \mu(\vr, \vt) \left( \Grad \vu + \Grad^t \vu - \frac{2}{N} \Div \vu \mathbb{I} \right) + \lambda (\vr, \vt) \Div \vu \mathbb{I};
\end{equation}
while the internal energy flux obeys \emph{Fourier's law}
\begin{equation} \label{N6}
\vc{q} = - \kappa(\vr, \vt) \Grad \vt.
\end{equation}

The system (\ref{N2}--\ref{N6}) is termed \emph{Navier--Stokes--Fourier system}.

\subsection{Boundary conditions}

We suppose that $\Omega \subset R^N$ is a bounded domain and consider the no-slip/no flux boundary conditions:
\begin{equation} \label{N7}
\vu|_{\partial \Omega} = 0,\ \vc{q} \cdot \vc{n}|_{\partial \Omega} = \kappa \Grad \vt \cdot \vc{n}|_{\partial \Omega} = 0.
\end{equation}

\subsection{The basic laws of thermodynamics}

The \emph{First law of thermodynamics} requires the total energy of the fluid system to be conserved:
\begin{equation} \label{N8}
\partial_t \left( \frac{1}{2} \vr |\vu|^2 + \vr e \right) + \Div \left[ \left( \frac{1}{2} \vr |\vu|^2 + \vr e \right) \vu \right] +
\Div (p(\vr, \vt) \vu) = \Div \left( \mathbb{S} \cdot \vu \right) - \Div \vc{q}.
\end{equation}
Note that (\ref{N8}) follows from (\ref{N2}--\ref{N4}) as long as all quantities are smooth.

The \emph{Second law of thermodynamics} postulates the entropy production principle:
\begin{equation} \label{N9}
\partial_t (\vr s) + \Div \left( \vr s \vu \right) + \Div \left( \frac{\vc{q}}{\vt} \right) = \sigma,
\end{equation}
where the entropy production rate $\sigma$ is given as
\begin{equation} \label{N10}
\sigma = \frac{1}{\vt} \left( \mathbb{S} : \Grad \vu - \frac{ \vc{q} \cdot \Grad \vt }{\vt} \right).
\end{equation}
Similarly to (\ref{N8}), the entropy balance equation (\ref{N9}), (\ref{N10}) can be derived directly from (\ref{N2}), (\ref{N4}), with help of Gibb's relation
(\ref{N1}), as soon as all quantities are smooth.

The rest of the paper is organized as follows. In Section \ref{G}, we introduce the concept of dissipative measure--valued solution to the Navier--Stokes--Fourier system. Section \ref{r} deals with the basic tool for studying stability -- the relative energy and the associated relative energy inequality --
expressed in terms of DMV solutions. In Section \ref{w}, we use the relative energy to compare the ``distance'' between a DMV solution and a smooth solution of the same problem. In Section \ref{u}, we establish the desired weak--strong uniqueness property under several restrictions imposed on both the DMV and the
strong solution. In Section \ref{C}, we find sufficient conditions for a DMV solution to be a classical one by applying the regularity criterion of
Sun, Wang, and Zhang \cite{SuWaZh}.
The paper is concluded by a short discussion on possible extensions of the theory in Section \ref{CC}.

\section{Generalized solutions}

\label{G}

The basic idea borrowed from \cite{FeNo6} is to retain equations (\ref{N2}), (\ref{N3}) while (\ref{N4}) is replaced by
\begin{itemize}
\item
total energy balance
\begin{equation} \label{G1}
\partial_t \intO{ \left( \frac{1}{2} \vr |\vu|^2 + \vr e \right) } = 0;
\end{equation}
\item
the entropy inequality
\begin{equation} \label{G2}
\begin{split}
\partial_t (\vr s) &+ \Div \left( \vr s \vu \right) + \Div \left( \frac{\vc{q}}{\vt} \right) = \sigma,\\
\sigma &\geq \frac{1}{\vt} \left( \mathbb{S} : \Grad \vu - \frac{ \vc{q} \cdot \Grad \vt }{\vt} \right).
\end{split}
\end{equation}

\end{itemize}

As shown in \cite[Chapter 2]{FeNo6}, any smooth weak (distributional) solution of (\ref{N2}), (\ref{N3}), (\ref{G1}), (\ref{G2}) satisfying the boundary
conditions (\ref{N7}) is a classical solution of the original system (\ref{N2}--\ref{N4}).
The basic theory in the framework of weak (distributional) solutions for the problem (\ref{N2}), (\ref{N3}), (\ref{G1}), (\ref{G2})
has been developed in \cite{FeNo6}. The main advantage of a weak formulation based on (\ref{G1}), (\ref{G2}) rather than (\ref{N8}) is that
the integrated energy balance (\ref{G1}) does not contain the convective terms present in (\ref{N8}), on which we have no control.

The \emph{measure--valued solutions} are typically generated by a sequence of weak solutions, see e.g. the monograph by M\' alek et al. \cite{MNRR}.
Here we adopt a slightly different approach advocated by Brennier, DeLellis, and Sz\' ekelyhidi \cite{BrDeSz}, Gwiazda et al. \cite{FGSWW1}, \cite{GSWW}
defining a measure--valued solution as an object independent of any generating procedure whatever the latter might be.

\subsection{Phase space}

A natural candidate for the phase space is of course the space based on the standard variables $[\vr, \vt, \vu]$. However, the entropy balance
(\ref{G2}), or more specifically, the entropy production rate $\sigma$, contains also spatial gradients of $(\vu, \vt)$.
A suitable phase space framework for the measure--valued solutions is therefore
\[
\mathcal{F} = \left\{ \left[\vr, \vt, \vu, \Du, \DT \right] \ \Big|
\ \vr \in [0, \infty), \ \vt \in [0, \infty), \ \vu \in R^N, \ \Du \in R^{N \times N}_{\rm sym}, \ \DT \in R^N \ \right\},
\]
where $\Du$ stands for the symmetric velocity gradient $\mathbb{D}[\vu] \equiv \frac{1}{2} \left(\Grad \vu + \Grad^t \vu \right)$, and $\DT$ for $\Grad \vt$. Oscillations in sequences of (weak) solutions are usually described in terms of the associated Young measure. Analogously, we define a \emph{measure--valued solution} to the Navier--Stokes--Fourier system
on the space--time cylinder $Q_T = (0,T) \times \Omega$
as a
family of \emph{probability measures} $\{ \Nu \}_{(t,x) \in Q_T}$,
\[
\Nu \in \mathcal{P}(\mathcal{F}) \ \mbox{for a.a.}\ (t,x) \in Q_T,\
(t,x) \in Q_T \mapsto \mathcal{M}(\mathcal{F}) \ \mbox{weakly-(*) measurable.}
\]
Indeed the measure $\Nu$ can be interpreted as probability that a solution $[\vr, \vt, \vu, \Du, \DT]$
at a point $(t,x) \in Q_T$ attains certain value in the phase space $\mathcal{F}$. Classical ``single valued'' solutions
$[\vr, \vt, \vu]$ are represented by a Dirac mass,
\[
[\vr, \vt, \vu, \Du, \DT] \approx \delta_{\vr, \vt, \vu, \mathbb{D}[\vu], \Grad \vt}.
\]

In addition, we require a very natural compatibility condition,
\[
\frac{1}{2} \left( \Grad \left< \Nu; \vu \right> + \Grad^t \left< \Nu; \vu \right> \right)= \left< \Nu; \Du \right>,\
\Grad \left< \Nu; \vt \right> = \left< \Nu; \DT \right> \ \mbox{in} \ \mathcal{D}'(Q_T),
\]
more precisely
\begin{equation} \label{G4}
\begin{split}
-  \int_0^T \intO{ \left< \Nu; \vu \right> \cdot \Div \mathbb{T} }\dt &=
\int_0^T \intO{\left< \Nu; \Du \right> : \mathbb{T} } \dt \ \mbox{for any}\ \mathbb{T} \in C^1(\Ov{Q}_T; R^{N \times N}_{\rm sym}), \\
 - \int_0^T \intO{ \left< \Nu; \vt \right>  \Div \bfphi }\dt &=
\int_0^T \intO{\left< \Nu; \DT \right> \cdot \bfphi } \dt \ \mbox{for any}\ \bfphi \in C^1(\Ov{Q}_T; R^N),\\ \bfphi \cdot \vc{n}|_{\partial \Omega} &= 0.
\end{split}
\end{equation}
Note that (\ref{G4}) also reflects  the no-slip condition $\vu|_{\partial \Omega} = 0$.

\subsection{Initial data}

Similarly to the preceding part, the initial state of the system
\[
[\vr_0, \vt_0, \vu_0] = [\vr(0, \cdot), \vt(0, \cdot), \vu(0, \cdot)]
\]
can be expressed by means of a parameterized family of measures $\{ \mathcal{V}_{0,x} \}_{x \in \Omega}$ acting on the phase space
\[
\mathcal{F}_0 = \left\{ [\vr_0, \vt_0, \vu_0] \ \Big| \vr_0 \in [0, \infty),\
\vt_0 \in [0, \infty), \vu_0 \in R^N \right\}.
\]
To avoid ambiguity concerning the values of $\vu_0$ on the vacuum, we always assume the initial density $\vr_0$
to be bounded below away from zero.

\subsection{Field equations}

In accordance with the previous discussion, a generalized formulation of the equation of continuity (\ref{N2}) reads
\begin{equation} \label{G5}
\left[ \intO{ \left< \Nu ; \vr \right> \varphi (t,x) } \right]_{t = 0}^{t = \tau} =
\int_0^\tau \intO{ \left[ \left< \Nu ; \vr \right> \partial_t \varphi(t,x) + \left< \Nu; \vr \vu \right> \cdot \Grad \varphi(t,x) \right] } \dt
\end{equation}
for a.a. $\tau \in (0,T)$ and any $\varphi \in \DC( [0,T) \times \Ov{\Omega})$.

Similarly, the momentum equation (\ref{N3}) reads
\begin{equation} \label{G6}
\begin{split}
&\left[ \intO{ \left< \Nu ; \vr \vu \right> \cdot \bfphi (t,x) } \right]_{t = 0}^{t = \tau} \\ &=
\int_0^\tau \intO{ \left[ \left< \Nu ; \vr \vu \right> \cdot \partial_t \bfphi(t,x) + \left< \Nu; \vr \vu \otimes \vu \right> : \Grad \bfphi(t,x)
+ \left< \Nu ; p(\vr, \vt) \right> \Div \bfphi (t,x) \right] } \dt\\
&+ \int_0^\tau \intO{ \left< \Nu; \mathbb{S}(\vr, \vt, \Du ) \right> : \Grad \bfphi(t,x) } \dt +
\int_0^\tau \int_{\Ov{\Omega}} \Grad \bfphi : {\rm d}{\nu_C}
\end{split}
\end{equation}
for a.a. $\tau \in (0,T)$ and any $\bfphi \in \DC([0,T) \times \Omega; R^N)$.
The symbol $\nu_C$ in (\ref{G6}) stands for the \emph{concentration measure}. More specifically, $\nu_C$ is a (tensor valued) signed Borel measure,\footnote{The term $\int_0^\tau \int_{\Ov{\Omega}} \Grad \bfphi : {\rm d}{\nu_C}$ is understood as the value of the functional $\nu_C$ over the continuous function $\Grad \bfphi$.}
$\nu_C \in \mathcal{M}( \Ov{Q}_T ; R^{N \times N})$, that will be controlled by the {\it concentration defect} introduced below, cf. also \cite{BreFei17}.

The energy equation (\ref{N8}) is replaced by the total energy balance (\ref{G1}), more specifically
\begin{equation} \label{G7}
\intO{ \left< \mathcal{V}_{\tau,x} ; \frac{1}{2} \vr |\vu|^2 + \vr e(\vr, \vt) \right>  }
\leq  \intO{ \left< \mathcal{V}_{0,x} ; \frac{1}{2} \vr |\vu|^2 + \vr e(\vr, \vt) \right>  }
\end{equation}
for a.a. $\tau \in (0,T)$.

Finally, we reformulate the entropy inequality (\ref{G2}) as
\begin{equation} \label{G8}
\begin{split}
&\left[ \intO{ \left< \Nu; \vr s(\vr, \vt) \right> \varphi(t,x) } \right]_{t = 0}^{t = \tau} \\& =
\int_0^\tau \intO{ \left[ \left< \Nu; \vr s(\vr, \vt) \right> \partial_t \varphi(t,x) +
\left< \Nu; \vr s(\vr, \vt) \vu - \frac{\kappa(\vr, \vt)}{\vt} \DT \right>
\cdot \Grad \varphi(t,x) \right] } \dt\\
&+ \int_0^\tau \int_{\Ov{\Omega}} \varphi (t,x) \ {\rm d}\sigma  \ \mbox{for a.a.}\ \tau \in (0,T),
\end{split}
\end{equation}
and any $\varphi \in \DC([0,T) \times \Ov{\Omega})$,
where the entropy production rate $\sigma \in \mathcal{M}^+(\Ov{Q}_T)$ is a non--negative Borel measure satisfying
\begin{equation} \label{G9}
\sigma \geq \left< \Nu ; \frac{1}{\vt} \left( \mathbb{S}(\vr, \vt, \Du) : \Du + \frac{\kappa (\vr, \vt)}{\vt} |\DT |^2 \right) \right>.
\end{equation}
Thus we may rewrite (\ref{G8}), (\ref{G9}) in a more concise form
\begin{equation} \label{G10}
\begin{split}
&\left[ \intO{ \left< \Nu; \vr s(\vr, \vt) \right> \varphi(t,x) } \right]_{t = 0}^{t = \tau} \\& \geq
\int_0^\tau \intO{ \left[ \left< \Nu; \vr s(\vr, \vt) \right> \partial_t \varphi(t,x) +
\left< \Nu; \vr s(\vr, \vt) \vu - \frac{\kappa(\vr, \vt)}{\vt} \DT \right>
\cdot \Grad \varphi(t,x) \right] } \dt\\
&+ \int_0^\tau \intO{ \left< \Nu ; \frac{1}{\vt} \left( \mathbb{S}(\vr, \vt, \Du) : \Du + \frac{\kappa (\vr, \vt)}{\vt} |\DT |^2 \right) \right>
\varphi (t,x) } \dt
\end{split}
\end{equation}
for a.a. $\tau \in (0,T)$ and any $\varphi \in \DC([0,T) \times \Ov{\Omega})$, $\varphi \geq 0$.

\begin{Remark} \label{R1}

The functions $(\vr, \vt, \Du) \mapsto \frac{1}{\vt} \mathbb{S}(\vr, \vt, \Du) : \Du $ and
$(\vr, \vt, \DT) \mapsto \frac{ \kappa (\vr, \vt) }{\vt^2} |\DT|^2$ are defined as
\[
\begin{split}
\frac{1}{\vt} \mathbb{S}(\vr, \vt, \Du) : \Du &= \lim_{\ep \to 0+}\frac{1}{\vt + \ep} \mathbb{S}(\vr, \vt, \Du) : \Du \in [0, \infty],\\
\frac{ \kappa (\vr, \vt) }{\vt^2} |\DT|^2 &=\lim_{\ep \to 0+ } \frac{ \kappa (\vr, \vt) }{(\vt + \ep)^2} |\DT|^2 \in [0, \infty]
\end{split}
\]
for any $\vt \in [0, \infty)$.

\end{Remark}

\subsection{Compatibility conditions}

Besides $(\ref{G4})$, further \emph{compatibility conditions} will be imposed on the measure--valued solution $\{ \Nu \}_{(t,x) \in Q_T}$. Note that
all compatibility conditions specified below (including  $(\ref{G4})$) are automatically satisfied as long as the parameterized measure $\{ \Nu \}_{(t,x) \in Q_T}$ is a Young measure generated by a family
$\{ \vre, \vte, \vue, \mathbb{D} [\Grad \vue], \Grad \vte \}_{\ep > 0}$ of functions satisfying the bounds
\begin{equation} \label{B}
\begin{split}
{\rm ess} &\sup_{t \in (0,T)} \intO{ \Big[ \vre + \vre |\vue|^2 + \vre e(\vre, \vte) + \vre |s(\vre, \vte)| \Big] } \\
&+ \int_0^T \intO{ \left[ \frac{\mu(\vre, \vte)}{2 \vte}\left| \Grad \vue + \Grad^t \vue - \frac{2}{N} \Div \vue \mathbb{I} \right|^2 +
\frac{\lambda (\vre, \vte)}{\vte}|\Div \vue |^2 \right] } \dt \\ &+
\int_0^T \intO{ \kappa (\vre, \vte) |\Grad \log(\vte)|^2 } \dt
\aleq 1
\end{split}
\end{equation}
uniformly for $\ep \to 0$.

Here and hereafter, the symbol $a \aleq b$ means $a \leq cb$ for a certain constant $c>0$.

\subsubsection{Concentration defect}
\label{CDC}

Set
\[
\mathcal{D}(\tau) = \intO{ \left< \mathcal{V}_{0,x}; \frac{1}{2} \vr |\vu|^2 + \vr e(\vr, \vt) \right> } -
\intO{ \left< \mathcal{V}_{\tau,x}; \frac{1}{2} \vr |\vu|^2 + \vr e(\vr, \vt) \right> }.
\]
The total energy balance (\ref{G7}) implies that $\mathcal{D} \geq 0$ a.a. in $(0,T)$. Similarly to
\cite{BreFei17}, we require
\begin{equation} \label{G11}
\int_0^T  \psi(t) \int_{\Ov{\Omega}} {\rm d}|\nu_C|   \aleq \int_0^T \psi(t) \mathcal{D}(t) \ \dt
\ \mbox{for any}\ \psi \in C[0,T], \ \psi \geq 0
\end{equation}
and $\mathcal{D} \in L^\infty(0,T)$.

\subsubsection{Korn--Poincar\' e inequality}

Let
\[
\mathbb{T}[\mathbb{A}] \equiv \mathbb{A} + \mathbb{A}^t - \frac{2}{N} {\rm tr}[\mathbb{A}] \mathbb{I}.
\]
Keeping in mind (\ref{G4}) we impose an analogue of the Korn--Poincar\' e inequality
\begin{equation} \label{G12}
\begin{split}
\int_0^\tau \intO{ \left< \Nu; \left| \vu - \tvu(t,x) \right|^2 \right> } \dt \aleq
\int_0^\tau \intO{ \left< \Nu; \left| \mathbb{T}[\Du] - \mathbb{T}[\Grad \tvu] \right|^2 \right> } \dt
\end{split}
\end{equation}
for any $\tvu \in L^2(0,T; W^{1,2}_0(\Omega; R^N))$. Note that $\mathbb{T}[\Du] = 2 \left( \Du - \frac{1}{N} {\rm tr}[\Du] \mathbb{I} \right)$.

The following lemma shows that the Korn--Poincar\' e inequality is naturally ``inherited'' by weakly converging Sobolev velocity fields vanishing on the boundary $\partial \Omega$.

\begin{Lemma} \label{L1}
Suppose that $\{ \vc{v}_\ep \}_{\ep > 0}$ is a sequence of functions bounded in $L^2(0,T; W^{1,2}_0(\Omega; R^N))$. Let
$\{ \mathcal{V}_{t,x} \}_{(t,x) \in Q_T}$ be a Young measure on the phase space $[R^N; R^{N \times N}]$ generated by $\{ \vc{v}_\ep , \Grad \vve \}$.
Then the Korn--Poincar\' e inequality (\ref{G12}) holds.

\end{Lemma}

\begin{proof}

As the fields $\vve$ vanish on $\partial \Omega$, they can be extended to be zero outside $\Omega$. Applying the standard Korn and Poincar\' e inequalities
we get
\[
\intO{ |\vve |^\beta } \aleq  \intO{ |\Grad \vve |^\beta } \aleq \intO{ |\Grad \vve + \Grad^t \vve - \frac{2}{N} \Div \vve \mathbb{I} |^\beta }
\]
uniformly in $\ep$ for any $1 < \beta < 2$.

Integrating in time and
letting $\ep \to 0$ we obtain
\[
\begin{split}
\int_0^\tau \intO{ \left< \Nu ; |\vv |^\beta \right> } \dt
 \aleq \int_0^\tau \intO{ \left< \Nu; |\mathbb{D}_v  - \frac{1}{N} {\rm tr}[ \mathbb{D}_v ] \mathbb{I} |^\beta \right> } \dt.
\end{split}
\]

We conclude the proof by letting $\beta \nearrow 2$ and Monotone convergence theorem.
\end{proof}

It is interesting to see that (\ref{G12}) follows from a weaker stipulation:
\begin{equation} \label{G12a}
\int_0^\tau \intO{ \left< \Nu; \left| \vu - \left< \Nu; \vu \right> \right|^2 \right> } \dt
\aleq
\int_0^\tau \intO{ \left< \Nu; \left| \mathbb{T}[\Du] - \left< \Nu; \mathbb{T}[\Du] \right> \right|^2 \right> } \dt,
\end{equation}
cf. a similar hypothesis in \cite{FGSWW1}. Indeed we have
\[
\begin{split}
\int_0^\tau &\intO{ \left< \Nu; \left| \vu - \tvu(t,x) \right|^2 \right> } \dt \\
&\aleq \int_0^\tau \intO{ \left< \Nu; \left| \vu - \left< \Nu; \vu \right> \right|^2 \right> } \dt
+ \int_0^\tau \intO{ \left< \Nu; \left| \left< \Nu; \vu \right> - \tvu(t,x) \right|^2 \right> } \dt\\
&= \int_0^\tau \intO{ \left< \Nu; \left| \vu - \left< \Nu; \vu \right> \right|^2 \right> } \dt
+ \int_0^\tau \intO{ \left| \left< \Nu; \vu \right> - \tvu(t,x) \right|^2  } \dt
\\
&\aleq \int_0^\tau \intO{ \left< \Nu; \left| \mathbb{T}[\Du] - \left< \Nu; \mathbb{T}[\Du] \right> \right|^2 \right> } \dt + \int_0^\tau \intO{ \left| \left< \Nu; \vu \right> - \tvu(t,x) \right|^2  } \dt.
\end{split}
\]
Finally, by the standard Korn--Poincar\' e inequality and (\ref{G4}),
\[
\int_0^\tau \intO{ \left| \left< \Nu; \vu \right> - \tvu(t,x) \right|^2  } \dt \aleq
\int_0^\tau \intO{ \left| \left< \Nu; \mathbb{T}[\Du] \right> - \mathbb{T}[\Grad \tvu](x) \right|^2 } \dt ;
\]
whence (\ref{G12}) follows as
\[
\left< \Nu; \left| \mathbb{T}[\Du] - \left< \Nu; \mathbb{T}[\Du] \right> \right|^2 \right> +
\left| \left< \Nu; \mathbb{T}[\Du] \right> - \mathbb{T}[\Grad \tvu](x) \right|^2 =
\left< \Nu; \left| \mathbb{T}[\Du] - \mathbb{T}[\Grad \tvu](x)  \right|^2 \right>.
\]

\subsection{Dissipative measure--valued solutions}

We are ready to introduce the concept of \emph{dissipative measure--valued (DMV) solutions} to the Navier--Stokes--Fourier system.

\begin{Definition} \label{D1}

We say that a parameterized family of probability measures $\{ \Nu \}_{(t,x) \in Q_T}$ is a \emph{dissipative measure--valued (DMV) solution}
to the Navier--Stokes--Fourier system (\ref{N2}--\ref{N7}), with the initial conditions $\{ \mathcal{V}_{0,x} \}_{x \in \Omega}$ if the following holds:

\begin{itemize}

\item $\{ \Nu \}_{(t,x) \in Q_T}$ is a weakly-(*) measurable family mapping the physical space $Q_T$ in the set of probability measures on
the phase space $\mathcal{F}$ satisfying the compatibility condition (\ref{G4}).

\item The integral identities (\ref{G5}--\ref{G7}), (\ref{G10}) hold true.

\item $\{ \Nu \}_{(t,x) \in Q_T}$ complies with the compatibility conditions (\ref{G11}), (\ref{G12a}).

\end{itemize}

\end{Definition}

\begin{Remark} \label{R22}

Alternatively, a DMV solution could be defined as a Young measure $\{ \Nu \}_{(t,x) \in Q_T}$ associated to a family $\{ \vre, \vte, \vue, \mathbb{D}[\Grad \vue], \Grad \vte \}_{\ep > 0}$
satisfying the uniform bound (\ref{B}). Then it is sufficient to require $\{ \Nu \}_{(t,x) \in Q_T}$ to fulfill the integral identities (\ref{G5}--\ref{G7}), (\ref{G10}). As already pointed out such a definition might be more restrictive than Definition \ref{D1} but provide better information concerning
the properties of $\{ \Nu \}_{(t,x) \in Q_T}$. We come back to this issue in the concluding Section \ref{CC}.

\end{Remark}

\section{Relative energy inequality}
\label{r}

Let
\[
H_{\tvt} = \vr \left[ e(\vr, \vt) - \tvt s(\vr, \vt) \right]
\]
be the ballistic free energy, cf. Ericksen \cite{Eri}.
Motivated by \cite{FeiNov10}, we introduce the relative energy
\begin{equation} \label{r1}
\begin{split}
\mathcal{E} &\left( \vr, \vt, \vu \Big|, \tvr, \tvt, \tvu \right)
\equiv \frac{1}{2} \vr |\vu - \tvu|^2 + \frac{\partial H_{\tvt} (\tvr, \tvt) }{\partial \vr}(\vr - \tvr) -
H_{\tvt} (\tvr, \tvt)\\&=
\left[ \frac{1}{2} \vr |\vu|^2 + \vr e(\vr, \vt) \right] - \left[ \frac{1}{2} \tilde \vr |\tilde \vu|^2 +
\tilde \vr e(\tilde \vr, \tilde \vt) \right] + \frac{1}{2} \tilde \vr |\tilde \vu|^2  - \vr \vu \cdot \tilde \vu + \frac{1}{2} \vr |\tilde \vu|^2 \\ &-
\tilde{\vt} \left( \vr s(\vr, \vt) - \tilde{\vr} s(\tilde \vr, \tilde \vt) \right) - \left( e(\tilde \vr, \tilde \vt) - \tilde \vt s( \tilde \vr, \tilde \vt) + \frac{p(\tilde \vr, \tilde \vt)}{\tilde \vr} \right) (\vr - \tilde \vr),
\end{split}
\end{equation}
where
the last equality was shown in \cite[Appendix]{BreFei17B}. It is
worth noting that $\mathcal{E}$ differs from the \emph{relative entropy} introduced in the classical paper by
Dafermos \cite{Daf4} by a multiplicative factor $\tvt$, see \cite{BreFei17B}. The advantage of working with relative energy rather than entropy is that
the present setting is compatible with the weak formulation (\ref{G5}--\ref{G10}).

Given a DMV solution $\{ \Nu \}_{(t,x) \in Q_T}$ of the Navier--Stokes--Fourier system, we consider
the quantity
\[
\intO{ \left< \Nu; \mathcal{E} \left( \vr, \vt, \vu \Big|, \tvr, \tvt, \tvu \right) \right> }.
\]
After a short inspection of all terms on the right--hand side of (\ref{r1}), it is easy to see that
\[
\left[ \intO{ \left< \Nu; \mathcal{E} \left( \vr, \vt, \vu \Big|, \tvr, \tvt, \tvu \right) \right> } \right]_{t = 0}^{t= \tau}
\]
can be expressed by means of the integral identities (\ref{G5}--\ref{G10}) as soon as the functions $(\tvr, \tvt, \tvu)$ are smooth and satisfy the
``compatibility conditions''
\[
\tvr>0, \ \tvt > 0, \ \tvu|_{\partial \Omega} = 0.
\]
After a bit tedious but straightforward manipulation that has been performed in detail in \cite{BreFei17} and \cite{FeiNov10}, we deduce from (\ref{r1}) and (\ref{G5}--\ref{G10}) the relative energy inequality:
\begin{equation} \label{r2}
\begin{split}
\Big[ &\intO{ \left< \Nu; \mathcal{E} \left( \vr, \vt, \vu \Big|, \tvr(t,x), \tvt(t,x), \tvu(t,x) \right) \right> } \Big]_{t = 0}^{t= \tau} + \mathcal{D}(\tau) \\
&+ \int_0^\tau \intO{ \left[ \left< \Nu; \frac{\tvt(t,x)}{\vt} \mathbb{S}(\vr, \vt, \Du) : \Du \right> + \tvt(t,x) \left< \Nu; \frac{\kappa (\vr, \vt)}{\vt^2} |\DT|^2
\right> \right] } \dt  \\
&- \int_0^\tau  \intO{ \left< \Nu; \mathbb{S}(\vr, \vt, \Du) \right> : \Grad \tvu(t,x)  } \dt
\\
&\leq- \int_0^\tau \intO{ \left[ \left< \Nu ; \vr s(\vr, \vt) \right> \partial_t \tilde \vt(t,x) +
\left< \Nu; \left( \vr s(\vr, \vt) \vu - \frac{\kappa(\vr, \vt)}{\vt} \DT \right) \right> \cdot \Grad \tilde \vt(t,x) \right] } \dt\\
&+ \int_0^\tau \intO{ \Big[ \left< \Nu ; \vr (\tvu (t,x)- \vu) \right> \cdot \partial_t \tilde \vu (t,x) +
\left< \Nu ; { \vr (\tvu(t,x) - \vu ) \otimes \vu }
\right> : \Grad \tilde \vu(t,x) \Big] } \dt \\ &- \int_0^\tau \intO{ \left< \Nu; p(\vr, \vt) \right> \Div \tvu(t,x) } \dt \\
&+ \int_0^\tau \intO{ \left[ \left< \Nu ; \vr \right> \partial_t \tvt(t,x) s(\tvr, \tvt)(t,x) + \left< \Nu ; \vr \vu \right> \cdot \Grad \tvt(t,x) s(\tvr, \tvt) (t,x) \right] } \dt\\
&+ \int_0^\tau \intO{ \left[ \left< \Nu ; \tvr(t,x) - \vr \right> \frac{1}{\tvr} \partial_t p (\tvr, \tvt)(t,x) - \left< \Nu ; \vr \vu \right> \cdot \frac{1}{\tvr} \Grad p(\tvr, \tvt)(t,x) \right] } \dt\\
&+ \int_0^\tau \int_{\Ov{\Omega}} \Grad \tvu(t,x) : {\rm d} \nu_C
\end{split}
\end{equation}
for any
\begin{equation} \label{r3}
\tvr, \tvt \in C^1(\Ov{\Omega}), \ \tvr, \tvt > 0, \ \tvu \in C^1([0,T] \times \Ov{\Omega}; R^N), \ \tvu|_{\partial \Omega} = 0.
\end{equation}

\section{Weak vs. strong solutions}
\label{w}

Following the approach of \cite{FeiNov10} we use the relative energy to compare a (DMV) solution with a strong solution starting from the same initial data. To this end, we introduce a cut--off function
\[
\psi_\delta \in \DC((0, \infty)^2), \ 0 \leq \psi_\delta \leq 1,\ \psi_\delta (\vr, \vt) = 1
\ \mbox{whenever}\ \delta < \vr < \frac{1}{\delta}\mbox{ and }  \delta < \vt < \frac{1}{\delta} \mbox{ for some } \delta >0.
\]
For $h = h(\vr, \vt, \vu, \Du, \DT)$ we consider a decomposition
\[
h = h_{\rm ess} + h_{\rm res}, \ h_{\rm ess} = \psi_\delta(\vr, \vt) h(\vr, \vt, \vu, \Du, \DT) , \ h_{\rm res} = (1 - \psi_\delta (\vr, \vt) ) h(\vr, \vt, \vu, \Du, \DT).
\]

Suppose that the thermodynamic functions are interrelated through (\ref{N1}), and, in addition, satisfy the \emph{hypothesis of thermodynamic stability}:
\begin{equation} \label{w1-}
\frac{\partial p(\vr, \vt)}{\partial \vr} > 0, \ \frac{\partial e(\vr, \vt)}{\partial \vt} > 0 \ \mbox{for any}\ \vr, \vt > 0.
\end{equation}
Under these circumstances,
it was shown in \cite[Chapter 3, Proposition 3.2]{FeNo6} that
\begin{equation} \label{w1}
\begin{split}
\mathcal{E}&\left(\vr, \vt, \vu \Big| \tvr, \tvt, \tvu \right) \\ &\geq c(\delta) \left\{
\left[ |\vr - \tvr|^2 + |\vt - \tvt|^2 + |\vu - \tvu|^2 \right]_{\rm ess}
+ \Big[ 1 + \vr + \vr|s(\vr,\vt)| + \vr e(\vr, \vt) + \vr |\vu|^2 \Big]_{\rm res}  \right\}
\end{split}
\end{equation}
whenever $2 \delta < \tvr < \frac{1}{\delta} - \delta$, $2 \delta <  \tvt < \frac{1}{\delta} - \delta$ for some $\delta > 0$. Accordingly, the function $\mathcal{E}\left(\vr, \vt, \vu \Big| \tvr, \tvt, \tvu \right)$ can be seen as a kind of distance between the triples $[\vr, \vt, \vu]$ and $[\tvr, \tvt, \tvu]$.

\subsection{Comparison between strong and (DMV) solutions}

Let $[\tvr, \tvt, \tvu]$ be a strong (classical) solution of the Navier--Stokes--Fourier system (\ref{N1}--\ref{N7}) on a time interval $[0,T]$. More specifically, the functions $[\tvr, \tvt, \tvu]$ belong to the class
\[
\tvr,\  \tvt  \in C^1 ([0,T] \times \Ov{\Omega}),\ \tvu \in C^1([0,T] \times \Ov{\Omega}; R^N), \ D^2_x \tvt,\ D^2_x \tvu
\in C([0,T] \times \Ov{\Omega})
\]
and satisfy (\ref{N1}--\ref{N7}), and (\ref{r3}). Set
\[
\vr_0(x) = \tvr (0, x), \ \vt_0 (x) = \tvt(0,x), \ \vu_0(x) = \tvu(0,x), \ x \in \Omega
\]
and consider the initial data for (DMV) solutions in the form
\begin{equation} \label{w2}
\mathcal{V}_{0,x} = \delta_{ \vr_0(x),  \vt_0(x), \vu_0(x) } , \ x \in \Omega.
\end{equation}
In what follows, we use the relative energy inequality (\ref{r2}) to compare a (DMV) solution $\{ \Nu \}_{(t,x) \in Q_T}$
emanating from the initial data (\ref{w2})
with the strong solution
$[\tvr, \tvt, \tvu]$. In view of (\ref{w1}), we fix $\delta > 0$ small enough so that
\[
0 < 2 \delta < \tvr(t,x), \tvt(t,x) < \frac{1}{\delta} - \delta \ \mbox{for all}\ (t,x) \in \Ov{Q}_T.
\]

Setting
\[
\mathcal{H}(\tau) \equiv \intO{ \left< \mathcal{V}_{\tau,x}; \mathcal{E} \left( \vr, \vt, \vu \Big|, \tvr(\tau,x), \tvt(\tau,x), \tvu(\tau,x) \right) \right> }  + \mathcal{D}(\tau),
\]
where $\mathcal{D}$ is the concentration defect identified in Section \ref{CDC},
and using (\ref{w2}), together with the compatibility condition (\ref{G11}), we deduce from (\ref{r2}) that
\begin{equation} \label{w3}
\begin{split}
\mathcal{H}(\tau) &+ \int_0^\tau \intO{ \left[ \left< \Nu; \frac{\tvt}{\vt} \mathbb{S}(\vr, \vt, \Du) : \Du \right> + \tvt \left< \Nu; \frac{\kappa (\vr, \vt)}{\vt^2} |\DT|^2
\right> \right] } \dt  \\
&- \int_0^\tau  \intO{ \left< \Nu; \mathbb{S}(\vr, \vt, \Du) \right> : \Grad \tvu  } \dt -
\int_0^\tau \intO{ \left< \Nu; \frac{\kappa(\vr, \vt)}{\vt} \DT \right> \cdot \Grad \tvt } \dt
\\
&\aleq- \int_0^\tau \intO{ \left[ \left< \Nu ; \vr s(\vr, \vt) \right> \partial_t \tilde \vt +
\left< \Nu; \vr s(\vr, \vt) \vu  \right> \cdot \Grad \tilde \vt \right] } \dt\\
&+ \int_0^\tau \intO{ \Big[ \left< \Nu ; \vr (\tvu - \vu) \right> \cdot \partial_t \tilde \vu  +
\left< \Nu ; { \vr (\tvu - \vu ) \otimes \vu }
\right> : \Grad \tilde \vu \Big] } \dt \\ &- \int_0^\tau \intO{ \left< \Nu; p(\vr, \vt) \right> \Div \tvu } \dt \\
&+ \int_0^\tau \intO{ \left[ \left< \Nu ; \vr \right> \partial_t \tvt s(\tvr, \tvt) + \left< \Nu ; \vr \vu \right> \cdot \Grad \tvt s(\tvr, \tvt)  \right] } \dt\\
&+ \int_0^\tau \intO{ \left[ \left< \Nu ; \tvr - \vr \right> \frac{1}{\tvr} \partial_t p (\tvr, \tvt) - \left< \Nu ; \vr \vu \right> \cdot \frac{1}{\tvr} \Grad p(\tvr, \tvt) \right] } \dt\\
&+ \int_0^\tau \mathcal{H}(t) \ \dt.
\end{split}
\end{equation}
The constant hidden in $\aleq$ may depend on the norm of the strong solution $[\tvr, \tvt, \tvu]$. Our goal in the remaining part of this section is to use the fact that $[\tvr, \tvt, \tvu]$ represent a strong solution of the Navier--Stokes--Fourier system to
rewrite (\ref{w3}) in a more concise form so that a Gronwall type argument may be applied to deduce $\mathcal{H} = 0$.

\subsubsection{Step 1}

Using the momentum equation
\[
\partial_t \tvu + \tvu \cdot \Grad \tvu = - \frac{1}{\tvr} \Grad p(\tvr, \tvt) + \frac{1}{\tvr} \Div \mathbb{S}(\tvr, \tvt, \Grad \tvu)
\]
we deduce
\[
\begin{split}
\Big< \Nu ; \vr &(\tvu - \vu) \Big> \cdot \partial_t \tilde \vu  +
\left< \Nu ; { \vr (\tvu - \vu ) \otimes \vu }
\right> : \Grad \tilde \vu \\ &=
\left< \Nu ; \vr (\tvu - \vu) \right> \cdot \left( \partial_t \tilde \vu + \tvu \cdot \Grad \tvu \right)  +
\left< \Nu ; { \vr (\tvu - \vu ) \otimes (\vu - \tvu) }
\right> : \Grad \tilde \vu\\
&\leq - \left< \Nu ; \vr (\tvu - \vu) \right> \cdot \left( \frac{1}{\tvr} \Grad p(\tvr, \tvt) - \frac{1}{\tvr} \Div \mathbb{S}(\tvr, \tvt, \Grad \tvu) \right)  +  \left< \Nu; \mathcal{E} \left(\vr, \vt, \vu \Big| \tvr, \tvt, \tvu \right) \right>\\
&= - \left< \Nu ; \vr (\tvu - \vu) \right> \cdot \frac{1}{\tvr} \Grad p(\tvr, \tvt) \\&+
\left< \Nu ; \left( \frac{\vr}{\tvr} - 1 \right) (\tvu - \vu) \right> \cdot \Div \mathbb{S}(\tvr, \tvt, \Grad \tvu)
+ \left< \Nu ; (\tvu - \vu) \right> \cdot \Div \mathbb{S}(\tvr, \tvt, \Grad \tvu)
\\
&+  \left< \Nu; \mathcal{E} \left(\vr, \vt, \vu \Big| \tvr, \tvt, \tvu \right) \right>.
\end{split}
\]

Next, using the compatibility condition (\ref{G4}) we may perform by parts integration
\[
\intO{ \left< \Nu ; (\tvu - \vu) \right> \cdot \Div \mathbb{S}(\tvr, \tvt, \Grad \tvu) } = -
\intO{ \left< \Nu ; (\SGrad [\tvu] - \Du) \right> :\mathbb{S}(\tvr, \tvt, \Grad \tvu) }.
\]

In view of the previous discussion, the relation (\ref{w3}) reduces to
\begin{equation} \label{w5}
\begin{split}
\mathcal{H}(\tau) &+ \int_0^\tau \intO{ \left[ \left< \Nu; \frac{\tvt}{\vt} \mathbb{S}(\vr, \vt, \Du) : \Du \right>
+ \mathbb{S}(\tvr, \tvt, \Grad \tvu ) : \Grad \tvu \right] } \dt  \\
&+
\int_0^\tau \intO{ \tvt  \left< \Nu; \kappa(\vr, \vt) \frac{\DT}{\vt} \cdot \left( \frac{\DT}{\vt} -  \Grad \log(\tvt) \right) \right> } \dt\\
&- \int_0^\tau  \intO{ \left< \Nu; \mathbb{S}(\vr, \vt, \Du) \right> : \Grad \tvu  } \dt
- \int_0^\tau  \intO{ \mathbb{S}(\tvr, \tvt, \Grad \tvu )  : \left< \Nu; \Du \right>  } \dt
\\
&\aleq- \int_0^\tau \intO{ \left[ \left< \Nu ; \vr \left[ s(\vr, \vt) - s(\tvr, \tvt) \right] \right> \partial_t \tilde \vt +
\left< \Nu; \vr \left[ s(\vr, \vt) - s(\tvr, \tvt) \right] \vu  \right> \cdot \Grad \tilde \vt \right] } \dt
\\ &- \int_0^\tau \intO{ \left< \Nu; p(\vr, \vt) \right> \Div \tvu } \dt -
\int_0^\tau \intO{ \left< \Nu; \vr \tvu \right> \cdot \frac{1}{\tvr} \Grad p(\tvr, \tvt) } \dt  \\
&+ \int_0^\tau \intO{ \left< \Nu ; \tvr - \vr \right> \frac{1}{\tvr} \partial_t p (\tvr, \tvt) } \dt\\
&+ \int_0^\tau \intO{ \left< \Nu; |\vr - \tvr |\ |\vu - \tvu | \right> } \dt +\int_0^\tau \mathcal{H}(t) \ \dt,
\end{split}
\end{equation}

\subsubsection{Step 2}

We have
\[
\begin{split}
&\left< \Nu; \vr \left[ s(\vr, \vt) - s(\tvr, \tvt) \right] \vu  \right> \cdot \Grad \tilde \vt =
\left< \Nu; \vr \left[ s(\vr, \vt) - s(\tvr, \tvt) \right]   \right> \tvu \cdot \Grad \tilde \vt \\ &+
\left< \Nu; \vr \left[ s(\vr, \vt) - s(\tvr, \tvt) \right] (\vu - \tvu)  \right> \cdot \Grad \tilde \vt.
\end{split}
\]

Next, we get
\[
\begin{split}
&\intO{ \left< \Nu; \vr \tvu \right> \cdot \frac{1}{\tvr} \Grad p(\tvr, \tvt) } \\ &=
\intO{ \left< \Nu; (\vr - \tvr)  \right> \tvu \cdot \frac{1}{\tvr} \Grad p(\tvr, \tvt) } +
\intO{ \tvu \cdot \Grad p(\tvr, \tvt) }\\
&=
\intO{ \left< \Nu; (\vr - \tvr)  \right> \tvu \cdot \frac{1}{\tvr} \Grad p(\tvr, \tvt) } -
\intO{ \Div \tvu \cdot p(\tvr, \tvt) }.
\end{split}
\]

Consequently, we rewrite (\ref{w5}) as
\begin{equation} \label{w6}
\begin{split}
\mathcal{H}(\tau) &+ \int_0^\tau \intO{ \left[ \left< \Nu; \frac{\tvt}{\vt} \mathbb{S}(\vr, \vt, \Du) : \Du \right>
+ \mathbb{S}(\tvr, \tvt, \Grad \tvu ) : \Grad \tvu \right] } \dt  \\
&+
\int_0^\tau \intO{ \tvt  \left< \Nu; \kappa(\vr, \vt) \frac{\DT}{\vt} \cdot \left( \frac{\DT}{\vt} -  \Grad \log(\tvt) \right) \right> } \dt\\
&- \int_0^\tau  \intO{ \left< \Nu; \mathbb{S}(\vr, \vt, \Du) \right> : \Grad \tvu  } \dt
- \int_0^\tau  \intO{ \mathbb{S}(\tvr, \tvt, \Grad \tvu )  : \left< \Nu; \Du \right>  } \dt
\\
&\aleq- \int_0^\tau \intO{  \left< \Nu ; \tvr \left[ s(\vr, \vt) - s(\tvr, \tvt) \right]_{\rm ess} \right> \left( \partial_t \tilde \vt +
 \tvu \cdot \Grad \tilde \vt \right)  } \dt
\\&+ \int_0^\tau \intO{ \left< \Nu; p(\tvr, \tvt) - p(\vr, \vt) \right> \Div \tvu } \dt   \\
&+ \int_0^\tau \intO{ \left< \Nu ; \tvr - \vr \right> \frac{1}{\tvr} \left( \partial_t p (\tvr, \tvt) + \tvu \cdot \Grad p(\tvr, \tvt) \right) } \dt\\
&+ \int_0^\tau \intO{ \left[ \left< \Nu; |\vr - \tvr |\ |\vu - \tvu | \right>
+ \left< \Nu; \vr \left| s(\vr, \vt) - s(\tvr, \tvt) \right| |\vu - \tvu|  \right> \right] }
\dt\\&+ \int_0^\tau \intO{ \left< \Nu; \vr \left| \left[ s(\vr, \vt) - s(\tvr, \tvt) \right]_{\rm res} \right| \right> } \dt \\&+ \int_0^\tau \mathcal{H}(t) \ \dt,
\end{split}
\end{equation}

\subsubsection{Step 3}

Using the equation of continuity (\ref{N2}), Gibbs' relation (\ref{N1}), and the entropy balance (\ref{N9}), (\ref{N10}), we deduce
\[
\begin{split}
\Big( p(\tvr, \tvt) &- p(\vr, \vt) \Big) \Div \tvu + (\tvr - \vr ) \frac{1}{\tvr} \left( \partial_t p (\tvr, \tvt) + \tvu \cdot \Grad p(\tvr, \tvt) \right)\\
&= \left( p(\tvr, \tvt) - \frac{\partial p(\tvr, \tvt) }{\partial \vr} (\tvr - \vr) - \frac{\partial p(\tvr, \tvt) }{\partial \vt} (\tvt - \vt)- p(\vr, \vt) \right) \Div \tvu\\ &- \left[
\tvr \frac{\partial s(\tvr, \tvt) }{\partial \vr} (\tvr - \vr) + \tvr \frac{\partial s(\tvr, \tvt) }{\partial \vt} (\tvt - \vt)
\right] \left( \partial_t \tilde \vt +
 \tvu \cdot \Grad \tilde \vt \right)\\
&+(\tvt - \vt) \Div \left( \kappa(\tvr, \tvt) \Grad \log \tvt \right) + \left( 1 - \frac{\vt}{\tvt} \right)
\left( \mathbb{S}(\tvr, \tvt, \Grad \tvu) : \Grad \tvu \right) + (\tvt - \vt) \kappa(\tvr, \tvt) |\Grad \log \tvt|^2.
\end{split}
\]
Plugging this in (\ref{w6}) we conclude
\begin{equation} \label{w7}
\begin{split}
\mathcal{H}(\tau) &+ \int_0^\tau \intO{ \left[ \left< \Nu; \frac{\tvt}{\vt} \mathbb{S}(\vr, \vt, \Du) : \Du \right>
+ \left< \Nu; \frac{\vt}{\tvt} \right> \mathbb{S}(\tvr, \tvt, \Grad \tvu ) : \Grad \tvu \right] } \dt  \\
&+
\int_0^\tau \intO{ \tvt  \left< \Nu; \kappa(\vr, \vt) \frac{\DT}{\vt} \cdot \left( \frac{\DT}{\vt} -  \Grad \log(\tvt) \right) \right> } \dt \\
&+ \int_0^\tau \intO{ \kappa (\tvr, \tvt) \Grad \log(\tvt) \cdot \left< \Nu; \vt \left( \Grad \log(\tvt) - \frac{\DT}{\vt} \right)  \right>   } \dt \\
&- \int_0^\tau  \intO{ \left< \Nu; \mathbb{S}(\vr, \vt, \Du) \right> : \Grad \tvu  } \dt
- \int_0^\tau  \intO{ \mathbb{S}(\tvr, \tvt, \Grad \tvu )  : \left< \Nu; \Du \right>  } \dt
\\
&\aleq \int_0^\tau \intO{ \left< \Nu; \left| p(\tvr, \tvt) - \frac{\partial p(\tvr, \tvt) }{\partial \vr} (\tvr - \vr) - \frac{\partial p(\tvr, \tvt) }{\partial \vt} (\tvt - \vt)- p(\vr, \vt) \right| \right> } \dt \\
&+ \int_0^\tau \intO{ \left[ \left< \Nu; |\vr - \tvr |\ |\vu - \tvu | \right>
+ \left< \Nu; \vr \left| s(\vr, \vt) - s(\tvr, \tvt) \right| |\vu - \tvu|  \right> \right] }
\dt\\
&+ \int_0^\tau \intO{ \left< \Nu;  \Big[ \vt +  \vr \left| \left[ s(\vr, \vt) - s(\tvr, \tvt) \right] \right| \Big]_{\rm res} \right> } \dt \\&+ \int_0^\tau \mathcal{H}(t) \ \dt.
\end{split}
\end{equation}

Finally, as a consequence of (\ref{w1}), we have
\begin{equation} \label{w8}
\begin{split}
\int_0^\tau &\intO{ \left< \Nu; \left| p(\tvr, \tvt) - \frac{\partial p(\tvr, \tvt) }{\partial \vr} (\tvr - \vr) - \frac{\partial p(\tvr, \tvt) }{\partial \vt} (\tvt - \vt)- p(\vr, \vt) \right| \right> } \dt \\
&+ \int_0^\tau \intO{ \left[ \left< \Nu; |\vr - \tvr |\ |\vu - \tvu | \right>
+ \left< \Nu; \vr \left| s(\vr, \vt) - s(\tvr, \tvt) \right| |\vu - \tvu|  \right> \right] }
\dt\\
&+ \int_0^\tau \intO{ \left< \Nu; \vr \left| \left[ s(\vr, \vt) - s(\tvr, \tvt) \right]_{\rm res} \right| \right> } \dt \\
&\aleq \int_0^\tau \intO{ \left< \Nu; \left[ \vt + |p(\vr, \vt)| + |\vu - \tvu| + \vr |s(\vr, \vt)| \ |\vu| \right]_{\rm res} \right> } \dt + \int_0^\tau \mathcal{H}(t) \ \dt.
\end{split}
\end{equation}
Thus the relations (\ref{w7}), (\ref{w8}) give rise to:
\begin{equation} \label{w9a}
\begin{split}
\mathcal{H}(\tau) &+ \int_0^\tau \intO{ \left[ \left< \Nu; \frac{\tvt}{\vt} \mathbb{S}(\vr, \vt, \Du) : \Du \right>
+ \left< \Nu; \frac{\vt}{\tvt} \right> \mathbb{S}(\tvr, \tvt, \Grad \tvu ) : \Grad \tvu \right] } \dt  \\
&- \int_0^\tau  \intO{ \left< \Nu; \mathbb{S}(\vr, \vt, \Du) \right> : \Grad \tvu  } \dt
- \int_0^\tau  \intO{ \mathbb{S}(\tvr, \tvt, \Grad \tvu )  : \left< \Nu; \Du \right>  } \dt
\\
&+
\int_0^\tau \intO{ \tvt  \left< \Nu; \kappa(\vr, \vt) \frac{\DT}{\vt} \cdot \left( \frac{\DT}{\vt} -  \Grad \log(\tvt) \right) \right> } \dt \\
&+ \int_0^\tau \intO{ \kappa (\tvr, \tvt) \Grad \log(\tvt) \cdot \left< \Nu; \vt \left( \Grad \log(\tvt) - \frac{\DT}{\vt} \right)  \right>   } \dt \\
&\aleq \int_0^\tau \intO{ \left< \Nu; \Big[ \vt + |p(\vr, \vt)| + |\vu - \tvu|+ \vr |s(\vr, \vt)| \ |\vu| \Big]_{\rm res} \right> } \dt + \int_0^\tau \mathcal{H}(t) \ \dt.
\end{split}
\end{equation}

Since  $\mathbb{S}$ is given by Newton's rheological law (\ref{N5}), we have
\[
\begin{split}
\mathbb{S}(\vr, \vt, \Du) : \Du &= \frac{\mu(\vr, \vt)}{2} \mathbb{T}[\Du]: \mathbb{T}[\Du] + \lambda(\vr, \vt) {\rm tr}^2[\Du], \\
\mathbb{S}(\tvr, \tvt, \Grad \tvu) : \Grad \tvu &= \frac{\mu(\tvr, \tvt)}{2} \mathbb{T}[\Grad \tvu]: \mathbb{T}[\Grad \tvu] + \lambda(\tvr, \tvt) (\Div \tvu)^2, \\
\mathbb{S}(\tvr, \tvt, \Grad \tvu) : \Du&= \frac{\mu(\tvr, \tvt)}{2} \mathbb{T}[\Grad \tvu]: \mathbb{T}[\Du] + \lambda(\tvr, \tvt) (\Div \tvu)({\rm tr}[\Du]),\\
 \mathbb{S}(\vr, \vt, \Du) : \Grad \tvu &= \frac{\mu(\vr, \vt)}{2} \mathbb{T}[\Grad \tvu]: \mathbb{T}[\Du] + \lambda(\vr, \vt) (\Div \tvu)({\rm tr}[\Du]),
\end{split}
\]
and hence  we can rewrite \eqref{w9a} as
\begin{equation}\label{w9b}
\begin{split}
\mathcal{H}&(\tau) + \int_0^\tau \intO{ \left[ \left< \Nu; \frac{\tvt \mu(\vr, \vt)}{2 \vt} \mathbb{T}[\Du] : \mathbb{T}[\Du] \right>
+ \left< \Nu; \frac{\vt \mu(\tvr, \tvt) }{2 \tvt} \right> \mathbb{T}[\Grad \tvu ] : \mathbb{T}[\Grad \tvu] \right] } \dt  \\
&- \int_0^\tau  \intO{ \left< \Nu; \frac{\mu (\vr, \vt)}{2} \mathbb{T}[\Du] \right> : \mathbb{T}[\Grad \tvu]  } \dt
- \int_0^\tau  \intO{ \frac{ \mu (\tvr, \tvt) }{2} \mathbb{T}[\Grad \tvu ]  : \left< \Nu; \mathbb{T}[\Du] \right>  } \dt
\\
&+ \int_0^\tau \intO{ \left[ \left< \Nu; \frac{\tvt \lambda(\vr, \vt)}{ \vt} {\rm tr}^2[\Du] \right>
+ \left< \Nu; \frac{\vt \lambda(\tvr, \tvt) }{ \tvt} \right> (\Div \tvu )^2 \right] } \dt  \\
&- \int_0^\tau  \intO{ \left< \Nu; \lambda (\vr, \vt) {\rm tr}[\Du] \right> \Div \tvu  } \dt
- \int_0^\tau  \intO{ \lambda (\tvr, \tvt)  \Div \tvu  \left< \Nu; {\rm tr}[\Du] \right>  } \dt
\\
&+
\int_0^\tau \intO{ \tvt  \left< \Nu; \kappa(\vr, \vt) \left|   \frac{\DT}{\vt} -  \Grad \log(\tvt) \right|^2 \right> } \dt \\
&+ \int_0^\tau \intO{ \kappa (\tvr, \tvt) \Grad \log(\tvt) \cdot \left< \Nu; (\vt - \tvt)\left( \Grad \log(\tvt) - \frac{\DT}{\vt} \right)  \right>   } \dt \\
&- \int_0^\tau \intO{   \Grad \tvt \cdot \left< \Nu; \big(\kappa(\vr, \vt) - \kappa (\tvr, \tvt)\big)\left( \Grad \log(\tvt) - \frac{\DT}{\vt} \right)  \right>   } \dt \\
&\aleq \int_0^\tau \intO{ \left< \Nu; \Big[ \vt + |p(\vr, \vt)| + |\vu - \tvu|+ \vr |s(\vr, \vt)| \ |\vu| \Big]_{\rm res} \right> } \dt + \int_0^\tau \mathcal{H}(t) \ \dt.
\end{split}
\end{equation}
Finally, we rewrite \eqref{w9b} in the form
\begin{equation}\label{w9}
\begin{split}
\mathcal{H}&(\tau) + \int_0^\tau \intO{  \left< \Nu;  \frac{ \mu(\vr, \vt)}{2 } \left| \sqrt{\frac{\tvt}{\vt}}  \mathbb{T}[\Du] - \sqrt{\frac{\vt}{\tvt}} \mathbb{T}[\Grad \tvu]\right|^2 \right>
 } \dt  \\
&- \frac12\int_0^\tau  \intO{ \mathbb{T}[\Grad \tvu] : \left< \Nu; \big(\mu (\vr, \vt) - \mu(\tvr,\tvt)\big)\left(\frac{\vt}{\tvt}   \mathbb{T}[\Grad \tvu] - \mathbb{T}[\Du]\right) \right> } \dt
\\
&+ \int_0^\tau \intO{  \left< \Nu;  \lambda(\vr, \vt) \left|\sqrt{\frac{\tvt }{ \vt}} {\rm tr}[\Du] -  \sqrt{\frac{\vt  }{ \tvt} } \Div \tvu  \right|^2 \right>} \dt  \\
&- \int_0^\tau  \intO{ \Div \tvu  \left< \Nu; \big(  \lambda (\vr, \vt) - \lambda (\tvr, \tvt)  \big)\left(  \frac{\vt}{\tvt} \Div \tvu - {\rm tr}[\Du] \right)  \right>  } \dt
\\
&+
\int_0^\tau \intO{ \tvt  \left< \Nu; \kappa(\vr, \vt) \left|   \frac{\DT}{\vt} -  \Grad \log(\tvt) \right|^2 \right> } \dt \\
&+ \int_0^\tau \intO{ \kappa (\tvr, \tvt) \Grad \log(\tvt) \cdot \left< \Nu; (\vt - \tvt)\left( \Grad \log(\tvt) - \frac{\DT}{\vt} \right)  \right>   } \dt \\
&- \int_0^\tau \intO{   \Grad \tvt \cdot \left< \Nu;\big(\kappa(\vr, \vt) - \kappa (\tvr, \tvt)\big) \left( \Grad \log(\tvt) - \frac{\DT}{\vt} \right)  \right>   } \dt \\
&\aleq \int_0^\tau \intO{ \left< \Nu; \Big[ \vt + |p(\vr, \vt)| + |\vu - \tvu|+ \vr |s(\vr, \vt)| \ |\vu| \Big]_{\rm res} \right> } \dt + \int_0^\tau \mathcal{H}(t) \ \dt.
\end{split}
\end{equation}

In the remaining part of this paper, we apply a Gronwall type argument to (\ref{w9}) to deduce $\mathcal{H} \equiv 0$ - the strong and the (DMV) solutions coincide in $Q_T$. This can be done under additional restrictions imposed either on the (DMV) solution and/or on the constitutive relations.
The first rather easy observation is that the first integral on the right-hand side of (\ref{w9}) vanishes as long as the ``residual'' component is empty.
More specifically, if there exists $\delta > 0$ such that
\[
\Nu \left\{ 0 < \delta < \vr < \frac{1}{\delta},\ 0 < \delta < \vt < \frac{1}{\delta} \right\} = 1
\ \mbox{for a.a.}\ (t,x) \in Q_T.
\]
Under these circumstances,
the left--hand side of (\ref{w9}) can be bounded from below as
\[
\begin{split}
 \int_0^\tau &\intO{  \left< \Nu;  \frac{ \mu(\vr, \vt)}{2 } \left| \sqrt{\frac{\tvt}{\vt}}  \mathbb{T}[\Du] - \sqrt{\frac{\vt}{\tvt}} \mathbb{T}[\Grad \tvu]\right|^2 \right>
 } \dt  \\
&- \frac12\int_0^\tau  \intO{ \mathbb{T}[\Grad \tvu] : \left< \Nu; \big(\mu (\vr, \vt) - \mu(\tvr,\tvt)\big)\left(\frac{\vt}{\tvt}   \mathbb{T}[\Grad \tvu] - \mathbb{T}[\Du]\right) \right> } \dt
\\
&+ \int_0^\tau \intO{  \left< \Nu;  \lambda(\vr, \vt) \left|\sqrt{\frac{\tvt }{ \vt}} {\rm tr}[\Du] -  \sqrt{\frac{\vt  }{ \tvt} } \Div \tvu  \right|^2 \right>} \dt  \\
&- \int_0^\tau  \intO{ \Div \tvu  \left< \Nu; \big(  \lambda (\vr, \vt) - \lambda (\tvr, \tvt)  \big)\left(  \frac{\vt}{\tvt} \Div \tvu - {\rm tr}[\Du] \right)  \right>  } \dt
\\
&+
\int_0^\tau \intO{ \tvt  \left< \Nu; \kappa(\vr, \vt) \left|   \frac{\DT}{\vt} -  \Grad \log(\tvt) \right|^2 \right> } \dt \\
&+ \int_0^\tau \intO{ \kappa (\tvr, \tvt) \Grad \log(\tvt) \cdot \left< \Nu; (\vt - \tvt)\left( \Grad \log(\tvt) - \frac{\DT}{\vt} \right)  \right>   } \dt \\
&- \int_0^\tau \intO{   \Grad \tvt \cdot \left< \Nu;\big(\kappa(\vr, \vt) - \kappa (\tvr, \tvt)\big) \left( \Grad \log(\tvt) - \frac{\DT}{\vt} \right)  \right>   } \dt \\
 &\ageq
- \int_0^\tau \intO{ \left< \Nu; |\vr - \tvr|^2 + |\vt - \tvt|^2 \right> } \dt .
\end{split}
\]
Thus applying Gronwall's lemma to the resulting expression we obtain the following result:

\begin{Theorem} \label{TG1}

Let the transport coefficients $\kappa = \kappa(\vr, \vt)$, $\mu(\vr, \vt)$, and $\lambda(\vr, \vt)$ be continuously differentiable and positive for
$\vr > 0, \ \vt > 0$.
Let the thermodynamic functions $p$, $e$, and $s$ be smooth for $\vr,\ \vt > 0$, satisfying Gibbs' relation (\ref{N1}) and the hypothesis of thermodynamic stability (\ref{w1-}).

Assume that $[\tvr, \tvt, \tvu]$ is a classical solution to the Navier--Stokes--Fourier system
(\ref{N2}--\ref{N10}) in $[0,T] \times \Omega$ emanating from the initial data $[\vr_0, \vt_0, \vu_0]$ with $\vr_0, \vt_0 > 0$ in $\Ov{\Omega}$. Assume further that $\{ \Nu \}_{(t,x) \in Q_T}$ is a (DMV) solution to the same problem in the sense of Definition \ref{D1}  such that
\[
\Nu \left\{ 0 < \delta < \vr < \frac{1}{\delta},\ 0 < \delta < \vt < \frac{1}{\delta} \right\} = 1
\ \mbox{for a.a.}\ (t,x) \in Q_T
\]
for some $\delta > 0$ and
\[
\mathcal{V}_{0,x} = \delta_{ [\vr_0, \vt_0, \vu_0] } \ \mbox{for a.a.}\ x \in \Omega.
\]

Then
\[
\mathcal{V}_{t,x} = \delta_{ [\tvr(t,x), \tvt(t,x), \tvu(t,x), \SGrad [\tvu] (t,x), \Grad \tvt(t,x)] } \ \mbox{for a.a.}\ (t,x) \in Q_T.
\]

\end{Theorem}

\section{Weak--strong uniqueness}
\label{u}

Our goal is to extend validity of Theorem \ref{TG1} to a larger class of (DMV) solutions.
For the sake of simplicity, we restrict ourselves to the case of vanishing bulk viscosity
\begin{equation} \label{u1}
\lambda \equiv 0.
\end{equation}

In the remaining part of this section, we impose a structural restriction
\begin{equation}\label{u2}
|p(\vr, \vt)| \aleq \left( 1 + \vr e(\vr, \vt) + \vr | s(\vr, \vt) | \right).
\end{equation}
Note that this is not a restrictive hypotheses, at least for gases, for which $p \approx \vr e$.
In accordance with (\ref{w1}), (\ref{u1}), (\ref{u2}), relation (\ref{w9}) reduces to
\begin{equation} \label{u3}
\begin{split}
\mathcal{H}&(\tau) + \int_0^\tau \intO{  \left< \Nu;  \frac{ \mu(\vr, \vt)}{2 } \left| \sqrt{\frac{\tvt}{\vt}}  \mathbb{T}[\Du] - \sqrt{\frac{\vt}{\tvt}} \mathbb{T}[\Grad \tvu]\right|^2 \right>
 } \dt  \\
&- \frac12\int_0^\tau  \intO{ \mathbb{T}[\Grad \tvu] : \left< \Nu; \big(\mu (\vr, \vt) - \mu(\tvr,\tvt)\big)\left(\frac{\vt}{\tvt}   \mathbb{T}[\Grad \tvu] - \mathbb{T}[\Du]\right) \right> } \dt
\\
&+
\int_0^\tau \intO{ \tvt  \left< \Nu; \kappa(\vr, \vt) \left|   \frac{\DT}{\vt} -  \Grad \log(\tvt) \right|^2 \right> } \dt \\
&+ \int_0^\tau \intO{ \kappa (\tvr, \tvt) \Grad \log(\tvt) \cdot \left< \Nu; (\vt - \tvt)\left( \Grad \log(\tvt) - \frac{\DT}{\vt} \right)  \right>   } \dt \\
&- \int_0^\tau \intO{   \Grad \tvt \cdot \left< \Nu;\big(\kappa(\vr, \vt) - \kappa (\tvr, \tvt)\big) \left( \Grad \log(\tvt) - \frac{\DT}{\vt} \right)  \right>   } \dt \\
&\aleq \int_0^\tau \intO{ \left< \Nu; \Big[ \vt  + |\vu - \tvu|+ \vr |s(\vr, \vt)| \ |\vu| \Big]_{\rm res} \right> } \dt + \int_0^\tau \mathcal{H}(t) \ \dt.
\end{split}
\end{equation}

\subsection{Conditional results}

We focus on the case with particular transport coefficients:
\begin{equation} \label{u4}
\kappa > 0  \ \mbox{constant}, \ \mu(\vt) = \mu_0 + \mu_1 \vt, \ \mu_0 > 0, \ \mu_1 \geq 0.
\end{equation}
In particular,
\[
\begin{split}
 &\frac{\mu_1}2  \left< \Nu;  \vt \left| \sqrt{\frac{\tvt}{\vt}}  \mathbb{T}[\Du] - \sqrt{\frac{\vt}{\tvt}} \mathbb{T}[\Grad \tvu]\right|^2 \right>
 - \frac{\mu_1}2 \mathbb{T}[\Grad \tvu] : \left< \Nu; (\vt - \tvt)\left(\frac{\vt}{\tvt}   \mathbb{T}[\Grad \tvu] - \mathbb{T}[\Du]\right) \right>
\\
&= \frac{\mu_1}2 \left< \Nu; \tvt \big|  \mathbb{T}[\Du] -  \mathbb{T}[\Grad \tvu]\big|^2 \right> -  \frac{\mu_1}2  \left< \Nu; (\vt - \tvt)\left( \mathbb{T}[\Du] - \mathbb{T}[\Grad \tvu] \right)\right>:\mathbb{T}[\Grad \tvu]
.
\end{split}
\]
Consequently, relation (\ref{u3}) simplifies as
\begin{equation} \label{u5}
\begin{split}
\mathcal{H}&(\tau)+ \int_0^\tau \intO{\left< \Nu; \mu_0 \left| \sqrt{ \frac{\tvt}{\vt} } \mathbb{T}[\Du] - \sqrt{ \frac{\vt}{\tvt} } \mathbb{T}[\Grad \tvu]
\right|^2  \right> } \dt    \\
&+ \int_0^\tau \intO{ \left< \Nu; \mu_1 \Big| \mathbb{T}[\Du] - \mathbb{T}[\Grad \tvu] \Big|^2 +
\kappa \left| \frac{\DT}{\vt} - \Grad \log(\tvt) \right|^2 \right> } \dt\\
&\aleq \int_0^\tau \intO{ \left< \Nu; \Big[ \vt^2 +  |\vu - \tvu|+ \vr |s(\vr, \vt)| \ |\vu| \Big]_{\rm res} \right> } \dt + \int_0^\tau \mathcal{H}(t) \ \dt
\ \mbox{for a.a.}\ \tau \in (0,T).
\end{split}
\end{equation}

\subsubsection{Constant transport coefficients}

We start with the simplest case when $\kappa$ and $\mu$ satisfy (\ref{u4}) with $\mu_1 = 0$.

\begin{Theorem} \label{T1}

Let $\kappa$ and $\mu$ be positive constants.
Let the thermodynamic functions $p$, $e$, and $s$ be smooth functions of $(\vr, \vt)$ satisfying Gibbs' relation (\ref{N1}), the hypothesis of thermodynamic stability (\ref{w1-}), together with (\ref{u2}).

Assume that $[\tvr, \tvt, \tvu]$ is a classical solution to the Navier--Stokes--Fourier system
(\ref{N2}--\ref{N10}) in $[0,T] \times \Omega$ emanating from the initial data $[\vr_0, \vt_0, \vu_0]$ with $\vr_0, \vt_0 > 0$ in $\Ov{\Omega}$. Assume further that $\{ \Nu \}_{(t,x) \in Q_T}$ is a (DMV) solution to the same problem in the sense of Definition \ref{D1}  such that
\begin{equation} \label{u6}
\Nu \left\{ \vt \leq \Ov{\vt}, \ |\vu| \leq \Ov{\vu} \right\} = 1 \ \mbox{for a.a.}\ (t,x) \in Q_T
\end{equation}
for some constants $\Ov{\vt}$, $\Ov{\vu}$ and
\[
\mathcal{V}_{0,x} = \delta_{ [\vr_0, \vt_0, \vu_0] } \ \mbox{for a.a.}\ x \in \Omega.
\]

Then
\[
\mathcal{V}_{t,x} = \delta_{ [\tvr(t,x), \tvt(t,x), \tvu(t,x), \SGrad [\tvu] (t,x), \Grad \tvt(t,x)] } \ \mbox{for a.a.}\ (t,x) \in Q_T.
\]

\end{Theorem}

\begin{proof}

By virtue of hypothesis (\ref{u6}), we have
\[
\left< \Nu; \Big[ \vt^2 +  |\vu - \tvu|+ \vr |s(\vr, \vt)| \ |\vu| \Big]_{\rm res} \right> \aleq \left< \Nu; \Big[ 1 + \vr |s(\vr, \vt)|  \Big]_{\rm res} \right>;
\]
whence, in accordance with (\ref{w1}), the inequality (\ref{u5}) reduces to
\[
\begin{split}
\mathcal{H}&(\tau)+ \int_0^\tau \intO{\left< \Nu; \mu_0 \left| \sqrt{ \frac{\tvt}{\vt} } \mathbb{T}[\Du] - \sqrt{ \frac{\vt}{\tvt} } \mathbb{T}[\Grad \tvu]
\right|^2  \right> } \dt
\\ &+ \int_0^\tau \intO{ \left< \Nu;
\kappa \left| \frac{\DT}{\vt} - \Grad \log(\tvt) \right|^2 \right> } \dt
\aleq  \int_0^\tau \mathcal{H}(t) \ \dt
\ \mbox{for a.a.}\ \tau \in (0,T).
\end{split}
\]
Thus the desired conclusion follows by applying the standard Gronwall lemma argument.

\end{proof}

\begin{Remark} \label{RR3}

Note that the conclusions of Theorems \ref{TG1} and \ref{T1} remain valid even if the hypothesis (\ref{G12a}) (the analogue of Korn--Poincar\' e inequality) is omitted. If, in addition, $\lambda = 0$, we may replace $\Du$ by its traceless component $\Du - \frac{1}{N} {\rm tr} [\Du] \mathbb{I}$.
This might be of interest when applying these results to certain numerical schemes for which the Korn--Poincar\' e inequality
does not hold.

\end{Remark}

\subsubsection{Temperature dependent viscosity}

In the case of temperature dependent viscosity, specifically if $\mu_1 > 0$, we may use the Korn--Poincar\' e inequality (\ref{G12}) in order to handle the velocity dependent terms on the right hand side of (\ref{u5}):
\[
\begin{split}
\int_0^\tau \intO{ \left< \Nu, | \ [\vu - \tvu]_{\rm res} \ | \right> } \dt  &\aleq
\int_0^\tau \intO{ \left< \Nu, \left[ \frac{1}{2 \delta} \right]_{\rm res}  + \delta |\vu - \tvu|^2 \right> } \dt \\
&\aleq c(\delta) \int_0^\tau \mathcal{H}(t) \dt + \delta \int_0^\tau \intO{ \left< \Nu;   |\vu - \tvu|^2 \right> } \dt.
\end{split}
\]
In view of (\ref{G12}), the last integral may be absorbed by the left hand side of (\ref{u5}). We thereby obtain:
\begin{equation} \label{u7}
\begin{split}
\mathcal{H}&(\tau)+ \int_0^\tau \intO{\left< \Nu; \mu_0 \left| \sqrt{ \frac{\tvt}{\vt} } \mathbb{T}[\Du] - \sqrt{ \frac{\vt}{\tvt} } \mathbb{T}[\Grad \tvu]
\right|^2  \right> } \dt    \\
&+ \int_0^\tau \intO{ \left< \Nu; \mu_1 \Big| \mathbb{T}[\Du] - \mathbb{T}[\Grad \tvu] \Big|^2 +
\kappa \left| \frac{\DT}{\vt} - \Grad \log(\tvt) \right|^2 \right> } \dt\\
&\aleq \int_0^\tau \intO{ \left< \Nu; \Big[ \vt^2 + \vr |s(\vr, \vt)| \ |\vu| \Big]_{\rm res} \right> } \dt + \int_0^\tau \mathcal{H}(t) \ \dt
\ \mbox{for a.a.}\ \tau \in (0,T).
\end{split}
\end{equation}

\begin{Theorem} \label{T2}
Let $\kappa$ and $\mu$ satisfy (\ref{u4}) with $\mu_1 > 0$.
Let the thermodynamic functions $p$, $e$, and $s$ comply with the perfect gas constitutive relations
\[
p(\vr, \vt) = \vr \vt, \ e(\vr, \vt) = c_v \vt,\ s(\vr, \vt) = \log \left( \frac{\vt^{c_v}}{\vr} \right) , \ c_v > 1.
\]

Assume that $[\tvr, \tvt, \tvu]$ is a classical solution to the Navier--Stokes--Fourier system
(\ref{N2}--\ref{N10}) in $[0,T] \times \Omega$ emanating from the initial data $[\vr_0, \vt_0, \vu_0]$ with $\vr_0, \vt_0 > 0$ in $\Ov{\Omega}$. Assume further that $\{ \Nu \}_{(t,x) \in Q_T}$ is a (DMV) solution to the same problem in the sense of Definition \ref{D1}  such that
\begin{equation} \label{u8}
\Nu \left\{ |s(\vr, \vt)| \leq \Ov{s} \right\} = 1 \ \mbox{for a.a.}\ (t,x) \in Q_T
\end{equation}
for a certain constant $\Ov{s}$ and
\[
\mathcal{V}_{0,x} = \delta_{ [\vr_0, \vt_0, \vu_0] } \ \mbox{for a.a.}\ x \in \Omega.
\]

Then
\[
\mathcal{V}_{t,x} = \delta_{ [\tvr(t,x), \tvt(t,x), \tvu(t,x), \SGrad [\tvu] (t,x), \Grad \tvt(t,x)] } \ \mbox{for a.a.}\ (t,x) \in Q_T.
\]

\end{Theorem}

\begin{proof}

In view of (\ref{w1}), (\ref{u7}), and (\ref{u8}), we have to control only the integral
\[
\int_0^\tau \intO{ \left< \Nu; [ \vt^2 ]_{\rm res} \right> }.
\]
It follows from (\ref{u8}) that
\[
\vt^{c_v} \aleq \vr; \ \mbox{whence}\ \vt^{c_v + 1} \aleq \vr \vt = \frac{1}{c_v} \vr e(\vr, \vt).
\]
Thus the desired conclusion follows from (\ref{w1}).

\end{proof}

\subsection{Unconditional results}

Let the molecular pressure  $p_M$ of the gas and the associate internal energy $e_M$ be interrelated through the monoatomic gas state equation,
\[
p_M(\vr, \vt) = \frac{2}{3} \vr e_M(\vr, \vt).
\]
It follows from Gibbs' relation (\ref{N1}) that
\begin{equation} \label{ZD10}
p_M(\vr, \vt) = \vt^{5/2} P \left( \frac{\vr}{\vt^{3/2}} \right)
\end{equation}
for some function $P$, see \cite[Chapter 2]{FeNo6}. Motivated by \cite[Chapter 3]{FeNo6}, we further assume that
$P \in C^1 [0, \infty) \cap C^5(0, \infty)$,
\begin{equation} \label{ZD11}
P(0) = 0, \ P'(q) > 0 \ \mbox{for all} \ q \geq 0,
\end{equation}
\begin{equation} \label{ZD12}
0 < \frac{ \frac{5}{3} P(q) - P'(q) q }{q} < c \ \mbox{for all}\ q > 0, \ \lim_{q \to \infty} \frac{P(q)}{q^{5/3}} = \Ov{p} > 0.
\end{equation}
Note that these requirements basically follow from the thermodynamic stability hypothesis (\ref{w1-}).

Moreover, we set
\begin{equation} \label{ZD13}
e_M(\vr, \vt) = \frac{3}{2} \vt \left( \frac{ \vt^{3/2} }{\vr} \right) P \left( \frac{\vr}{\vt^{3/2}} \right),
\end{equation}
and, by virtue of Gibbs' relation (\ref{N1}),
\begin{equation} \label{ZD14}
s_M (\vr, \vt) = S \left( \frac{\vr}{\vt^{3/2}} \right),
\end{equation}
for some function $S$, see \cite[Chapter 2]{FeNo6}, where
\begin{equation} \label{ZD15}
S'(q) = - \frac{3}{2} \frac{ \frac{5}{3} P(q) - P'(q) q }{q^2} < 0.
\end{equation}
Finally, we impose the Third law of thermodynamics in the form
\begin{equation} \label{ZD16}
\lim_{q \to \infty} S(q) = 0.
\end{equation}

\begin{Theorem} \label{T3}
Let $\kappa$ and $\mu$ satisfy (\ref{u4}) with $\mu_1 > 0$.
Let the thermodynamic functions $p$, $e$, and $s$ be given as
\[
p(\vr, \vt) = p_M(\vr, \vt) + p_R(\vr, \vt), \ e(\vr, \vt) = e_M(\vr, \vt) + e_R(\vr, \vt),\
s(\vr, \vt) = s_M(\vr, \vt) + s_R(\vr, \vt),
\]
where $p_M$, $e_M$, and $s_M$ satisfy (\ref{ZD10}--\ref{ZD16}), and
\begin{equation} \label{ZZ1}
p_R = a \vt^2, \ e_R = a \frac{\vt^2}{\vr}, \ s_R = 2 a \frac{\vt}{\vr}, \ a > 0.
\end{equation}

Assume that $[\tvr, \tvt, \tvu]$ is a classical solution to the Navier--Stokes--Fourier system
(\ref{N2}--\ref{N10}) in $[0,T] \times \Omega$ emanating from the initial data $[\vr_0, \vt_0, \vu_0]$ with $\vr_0, \vt_0 > 0$ in $\Ov{\Omega}$. Assume further that $\{ \Nu \}_{(t,x) \in Q_T}$ is a (DMV) solution to the same problem in the sense of Definition \ref{D1}  such that
\[
\mathcal{V}_{0,x} = \delta_{ [\vr_0, \vt_0, \vu_0] } \ \mbox{for a.a.}\ x \in \Omega.
\]

Then
\[
\mathcal{V}_{t,x} = \delta_{ [\tvr(t,x), \tvt(t,x), \tvu(t,x), \SGrad [\tvu] (t,x), \Grad \tvt(t,x)] } \ \mbox{for a.a.}\ (t,x) \in Q_T.
\]

\end{Theorem}

\begin{Remark} \label{R2}

The quantity $p_R$ corresponds to the ``radiation'' component of the pressure. Note that the standard theory of radiative fluids postulates
$p_R = a \vt^4$ rather than (\ref{ZZ1}), see Oxenius \cite{OX} or \cite[Chapter 2]{FeNo6}. The present setting is in the spirit of models
of neutron stars studied by Lattimer et al.  \cite{LaVaPrPr}.

\end{Remark}

\begin{proof}

A short inspection of (\ref{u7}) reveals that the only term we have to handle is
\[
\int_0^\tau \intO{ \left< \Nu; \left[ \vr |s_M(\vr, \vt)|^2 + \vr s_R |\vu| \right]_{\rm res} \right> }.
\]

First, it follows from hypothesis (\ref{ZZ1}) that
\[
\vr s_R |\vu| \aleq \vr s_R |\vu - \tvu | + \vr s_R\aleq \vt |\vu - \tvu | + \vr s_R
\]
and
\[ \vt |\vu - \tvu | \leq c(\delta) \vt^2 + \delta   |\vu - \tvu |^2 \aleq c(\delta) \vr e_R + \delta   |\vu - \tvu |^2.
\]

Second, it was shown in \cite[Section 3.5.1]{BreFei17A} that the inequality
\[
\vr |s_M (\vr, \vt) |^2 \aleq 1 + \vr + \vr e_M(\vr, \vt)
\]
follows from the hypotheses (\ref{ZD10}--\ref{ZD16}).

We complete the proof by applying  (\ref{G12}) and (\ref{w1}).
\end{proof}

\section{Conditional regularity}
\label{C}

Combining the regularity criterion of Sun, Wang, and Zhang \cite{SuWaZh} with the (DMV)--strong uniqueness result stated in Theorem \ref{T1}
we deduce the following result.

\begin{Theorem} \label{T4}
Let $\kappa > 0$, $\mu > 0$, and $\lambda \geq 0$ be constant. Let the thermodynamic functions $p$, $e$, and $s$ comply with the perfect gas constitutive relations
\[
p(\vr, \vt) = \vr \vt, \ e(\vr, \vt) = c_v \vt,\ s(\vr, \vt) = \log \left( \frac{\vt^{c_v}}{\vr} \right) , \ c_v > 1.
\]

Assume that $\{ \Nu \}_{(t,x) \in Q_T}$ is a (DMV) solution of the Navier--Stokes--Fourier system
(\ref{N2}--\ref{N10}) in the sense of Definition \ref{D1} such that
\begin{equation} \label{C1}
\Nu \left\{ 0 < \underline{\vr} \leq \vr \leq \Ov{\vr}, \ \vt \leq \Ov{\vt}, \ |\vu| \leq \Ov{u} \right\} = 1
\ \mbox{for a.a.} \ (t,x) \in Q_T
\end{equation}
for some constants $\underline{\vr}$, $\Ov{\vr}$, $\Ov{\vt}$, and $\Ov{\vu}$. Assume further that $\Omega \subset R^3$ is a bounded domain of class $C^{2 + \nu}$, and
\[
\mathcal{V}_{0,x} = \delta_{\vr_0(x), \vt_0(x), \vu_0(x)} \ \mbox{for a.a.}\ x \in \Omega,
\]
where $[\vr_0, \vt_0, \vu_0]$ belong to the regularity class
\[
\begin{split}
\vr_0, \vt_0 &\in W^{3,2}(\Omega), \ \vr, \ \vt > 0 \ \mbox{in} \ \Ov{\Omega},\  \vu_0 \in W^{3,2}_0 (\Omega; R^3)\\
\vu_0|_{\partial \Omega} &= 0, \ \Grad \vt_0 \cdot \vc{n}|_{\partial \Omega} = 0,\
\Grad p(\vr_0, \vt_0) = \Div \mathbb{S}( \Grad \vu_0).
\end{split}
\]

Then
\[
\mathcal{V}_{t,x} = \delta_{ [\tvr(t,x), \tvt(t,x), \tvu(t,x), \SGrad [\tvu] (t,x), \Grad \tvt(t,x)] } \ \mbox{for a.a.}\ (t,x) \in Q_T,
\]
where $[\tvr, \tvt, \tvu]$ is a classical solution to the Navier--Stokes--Fourier system with the initial data $[\vr_0, \vt_0, \vu_0]$.

\end{Theorem}

\begin{proof}

By virtue of the standard local existence results of Valli, and Valli and Zajaczkowski \cite{Vall2}, \cite{Vall1}, and \cite{VAZA} (among others),
the Navier--Stokes--Fourier system (\ref{N2}--\ref{N7}) admits a local classical solution $[\tvr, \tvt, \tvu]$ on a maximal existence interval $[0, T_{\rm max})$, $T_{\rm max} > 0$.
It follows from  hypothesis (\ref{C1}) and Theorem \ref{T1}, that the (DMV) solution coincides with $[\tvr, \tvt, \tvu]$ at least on any compact subinterval of
$[0, T_{\rm max})$.

On the other hand, it follows from hypothesis (\ref{C1})
and the fact that the strong and (DMV) solutions coincide that the classical solution $[\tvr, \tvt, \tvu]$ complies with the regularity criterion established by
Sun, Wang, and Zhang \cite{SuWaZh}. Consequently, the classical solution $[\tvr, \tvt, \tvu]$ can be extended up to the time $T$, and, by Theorem \ref{T1} coincides with the (DMV) solution on $Q_T$.

\end{proof}

\section{Concluding remarks}
\label{CC}

We discuss briefly the case of more complex constitutive equations considered in the context of weak (distributional) solutions in \cite{FeiNov10}. The thermodynamic functions have the same structure as in Theorem \ref{T3}:
\[
p(\vr, \vt) = p_M(\vr, \vt) + p_R(\vr, \vt), \ e(\vr, \vt) = e_M(\vr, \vt) + e_R(\vr, \vt),
\]
where the molecular components $p_M$, $e_M$, (and $s_M$) obey (\ref{ZD10}--\ref{ZD16}) while the radiation terms satisfy
\[
p_R = a \frac{1}{3} \vt^4,\ e_R = a \frac{\vt^4}{\vr}, \ a > 0.
\]
The transport coefficients $\mu = \mu(\vt)$, $\eta = \eta(\vt)$ and $\kappa = \kappa(\vt)$ are continuously differentiable functions of the temperature $\vt$,
\[
| \mu'(\vt) | \aleq c, \ (1 + \vt^\alpha) \aleq \mu(\vt) \leq (1 + \vt^\alpha) ,\ 0 \leq \eta(\vt) \aleq (1 + \vt^\alpha),\
\alpha \in \left( \frac{2}{5}; 1 \right],
\]
\[
(1 + \vt^3) \aleq \kappa(\vt) \aleq (1 + \vt^3).
\]

The above hypotheses guarantee the \emph{existence} of a weak (distributional) solution to the Navier--Stokes--Fourier system, see \cite[Chapter 3]{FeNo6}.
Moreover, as shown in \cite{Feireisl2012}, \cite{FeiNov10}, the weak--strong uniqueness principle holds in this class of weak solutions. The proof requires
validity of certain Sobolev embedding relations that can be verified for the weak solutions but that might be violated by the (DMV) solutions specified
in Definition \ref{D1}. To retain these properties and, consequently, to extend the results of \cite{FeiNov10} to the class of (DMV) solutions,
the alternative definition proposed in Remark \ref{R22} must be adopted. We leave the details to the interested reader.

\def\cprime{$'$} \def\ocirc#1{\ifmmode\setbox0=\hbox{$#1$}\dimen0=\ht0
  \advance\dimen0 by1pt\rlap{\hbox to\wd0{\hss\raise\dimen0
  \hbox{\hskip.2em$\scriptscriptstyle\circ$}\hss}}#1\else {\accent"17 #1}\fi}

%\bibliographystyle{plain}
%\bibliography{citace}

\end{document}